\documentclass[12pt]{amsart}
\usepackage{amssymb,amsmath,amsthm,comment}
\usepackage{pdfpages,graphicx} 
\usepackage[section]{placeins} 
\usepackage{wrapfig}
\usepackage{float}
\usepackage{lineno}
\usepackage{youngtab}

\usepackage{setspace}
\newcommand{\shrinkmargins}[1]{
  \addtolength{\textheight}{#1\topmargin}
  \addtolength{\textheight}{#1\topmargin}
  \addtolength{\textwidth}{#1\oddsidemargin}
  \addtolength{\textwidth}{#1\evensidemargin}
  \addtolength{\topmargin}{-#1\topmargin}
  \addtolength{\oddsidemargin}{-#1\oddsidemargin}
  \addtolength{\evensidemargin}{-#1\evensidemargin}
  }
\shrinkmargins{0.5}

\newtheorem{theorem}{Theorem}
\newtheorem{lemma}[theorem]{Lemma}

\newtheorem{corollary}[theorem]{Corollary}
\newtheorem*{theorem*}{Theorem}

\newtheorem{conjecture}{Observation}

{Claim}

\newtheorem*{thm12}{Theorem~\ref{th03}}
\newtheorem*{thm13}{Theorem~\ref{th12}}
\theoremstyle{definition}

\theoremstyle{remark}

\newtheorem*{remarks}{{\bf Remarks}}

\newtheorem*{Acknowledgments}{Acknowledgments}
\numberwithin{theorem}{section} \numberwithin{equation}{section}

\usepackage{caption}
\usepackage{multirow}

\def\func#1{\mathop{\rm #1}}%

\begin{document}
\title[Log-Concave Infinite Products]{Log-Concavity of Infinite Product Generating Functions}
\author{Bernhard Heim }
\address{Lehrstuhl A f\"{u}r Mathematik, RWTH Aachen University, 52056 Aachen, Germany}
\email{bernhard.heim@rwth-aachen.de}
\author{Markus Neuhauser}
\address{Kutaisi International University, 5/7, Youth Avenue,  Kutaisi, 4600 Georgia}
\address{Lehrstuhl A f\"{u}r Mathematik, RWTH Aachen University, 52056 Aachen, Germany}
\email{markus.neuhauser@kiu.edu.ge}
\subjclass[2010] {Primary 05A17, 11P82; Secondary 05A20}
\keywords{Generating functions, Log-concavity, Partition numbers.}
\begin{abstract}
In the $1970$s Nicolas proved that the coefficients $p_d(n)$ defined by the generating function
\begin{equation*}
\sum_{n=0}^{\infty} p_d(n) \, q^n = \prod_{n=1}^{\infty} \left( 1- q^n\right)^{-n^{d-1}}
\end{equation*}
are log-concave for $d=1$. Recently, Ono, Pujahari, and Rolen have extended the result to $d=2$.
Note that $p_1(n)=p(n)$ is the partition function and
$p_2(n)=\func{pp}\left( n\right) $ is the number of plane partitions.
In this paper, we invest in properties for $p_d(n)$ for general $d$. Let $n \geq 6$.
Then $p_d(n)$ is almost log-concave for $n$ divisible by $3$ and almost strictly log-convex otherwise.
\end{abstract}
\maketitle
\section{Introduction and main results}
Many partition functions have infinite products as generating function.
Numerous researchers have investigated
their asymptotic behaviors and combinatorial properties.
Let $ \alpha:= (\alpha_n)_n$ be a sequence of
positive integers. In this paper, we investigate
the log-concavity of
the coefficients $p_{\alpha}(n)$ defined by \cite{An98}, chapter $6$:
\begin{equation*}
1+ \sum_{n=1}^{\infty} p_{\alpha}(n) \, q^n := \prod_{n \in S} 
\left( 1- q^n \right)^{-\alpha_n}, \qquad S\subset \mathbb{N}, \,\,  \vert q \vert < 1.
\end{equation*}
We focus on $S=\mathbb{N}$ and the sequences 
$\alpha_n(d):= n^{d-1}$, $(d \in \mathbb{N})$ and denote the coefficients by $p_d(n)$.
A sequence of non-negative integers $
\left( a_{n}\right) _{n}
$ is considered to be log-concave
at $n$ (\cite{St89, Br89} and \cite{OPR22} introduction) if
\begin{equation*}
\Delta(n):= a_n^2-a_{n-1} \, a_{n+1} \geq 0.
\end{equation*}
%
Most prominent is {\it the} partition function
$p(n)$, counting all non-ordered decompositions
of a natural number $n$ as
a sum of positive integers.
Euler proved the identity $p(n)=p_1(n)$, the famous
Eulerian infinite product identity
$$\sum_{n=0}^{\infty} p(n) \, q^n = \prod_{n=1}^{\infty} \frac{1}{1 - q^n}.$$
Building on the celebrated work of Hardy and Ramanujan \cite{HR18} on the asymptotic formula
\begin{equation*}
p(n) \sim \frac{1}{4 n \sqrt{3}}\,\, e^{\pi \sqrt{2n/3}}
\end{equation*}
and Lehmer \cite{Le39} and Rademacher \cite{Ra37}, Nicolas \cite{Ni78}, in 1978, proved
that $p(n)$ is log-concave if and only if
$n$ is an even integer or $n$ is odd and $n >25$. We refer to recent work
by DeSalvo and Pak \cite{DP15}.
Another prominent partition function is provided by the 
plane partitions $\func{pp}\left( n\right)$. For further information we refer to Andrews, 
Krattenthaler, and Stanley \cite{An98,
Kr16, St89} and Section~\ref{abschnitt2}.
A century ago, MacMahon  \cite{Ma99, Ma60} proved the 
fundamental identity $\func{pp}\left( n\right) = p_2(n)$.
Improving on the works of E. M. Wright \cite{Wr31}
on asymptotic formulas of $\func{pp}\left( n\right) $, 
Ono, Pujahari, and Rolen \cite{OPR22} finally proved that $\func{pp}\left( n\right) $ is log-concave if and only if
$n$ is an even integer or $n$ is odd and $n>11$. 
Previously it was known \cite{HNT21}, that $\func{pp}\left( n\right) $ is 
log-concave for almost all $n$ and tested up to $10^5$.

Supplied with these results
we analyze the log-concavity patterns for all $d\geq 1$.
Table \ref{table1} displays the results for
$1 \leq d \leq 8$.
\begin{table}[H]
\begin{tabular}{r|l|l|l}
\hline
$d$&log-concave&strictly log-convex&proof / verification\\ \hline \hline
$1$&$n > 25 , 1 < n < 25$ even&$1 \leq n \leq 25$ odd&1978 Nicolas ($ n \geq 1$)\\
$2$&$n > 11,   1 < n < 11$ even &$1 \leq n \leq 11$ odd&2022 Ono et al.\ ($n \geq 1$)\\
$3$&$n >7 , 1 < n < 7$ even&$1 \leq n \leq 7$ odd &2022 H--N ($n \leq 10^{5}$)\\
$4$&$n >5, \{2,3,4\}$&$\{1,5\}$& 
\phantom{xxxxxx}\vdots
\\
$5$&$n >5, \{2,3,4\}$&$\{1,5\}$&  
\phantom{xxxxxx}\vdots
\\
$6$&$n >5, \{2,3\}$&$\{1,4,5\}$&   \phantom{xxxxxx}\vdots
\\
$7$&$n >5, \{2,3\}$&$\{1,4,5\}$&   2022 H--N ($n \leq 10^{5
}$) 
\\
$8$&$n >5, \{2,3\}$&$\{1,4,5\}$&2022 H--N ($n \leq 10^{4}$)\\
\hline
\end{tabular}
\caption{\label{table1}}
\end{table}
Let $d$ be fixed. We
call a natural number $n$ an exception if
$\Delta_d(n)=\left( p_{d}\left( n\right) \right) ^{2}-p_{d}\left( n-1\right) p_{d}\left( n+1\right) <0$.
This indicates that $p_d(n)$ is strictly log-convex at $n$. It is obvious that $n=1$ is an exception for
all $d$. Moreover, Table \ref{table1} indicates that $n=4$ and $n=5$ are also candidates for  $\Delta_d(n) <0$ for $d>4$.
It turns out that this is true for $n \geq 6$ coprime to $3$ and almost all $d$. The bounds depend on $n$.
Thus, let us first study $\Delta_d(n)$ for small $n$.
%
\begin{theorem}\label{EINS}
Let $1 \leq n \leq 9$. 
Then for all $d \in \mathbb{N}$ we have:
\begin{equation*}
\begin{array}{|c|c|c|}
\hline
n &  \text{ strictly log-convex if and only if }
\\
\hline
1 &  d \geq 1 \\
2 & \emptyset \\
3 & 1 \leq d \leq 3 \\
4 & d \geq 6 \\
5 & 1 \leq d \leq 9\\
6 & \emptyset \\
7 &  1 \leq d \leq 3, d \geq 11 \\
8  & d \geq 9 \\
9  & 1 \leq d \leq 2\\
\hline
\end{array}
\end{equation*}
\end{theorem}
Obviously, Theorem \ref{EINS} illustrates some irregular behavior of $\Delta_d(n)$.
Nevertheless, there is a dominant pattern for each $n \geq 6$, coprime to $3$, forcing
$\Delta_d(n)$ to be negative for almost all $d$. We also have an interesting result for $n$ divisible by $3$.
\begin{theorem}\label{th03}
\label{th0}Let $n \geq 6$
be divisible by $3$. Then
$\Delta_d(n)>0$ for almost all $d$.
\end{theorem}
\begin{remarks}\ \\ 
Let $n \geq 6$ be divisible by $3$. Nicolas states that
partitions $p(n)$ are not log-concave if and only
if $n=9, 15$,
and $n=21$ (the case $d=1$).
Ono, Pujahari, and Rolen's result states, that plane partitions
$\func{pp}\left( n\right) $ are
not log-concave if and only
if $n=9$ (the case $d=2$).
Further, for $n \leq 6$ we have only $\Delta_3(3)< 0$.
Moreover, it is much more likely that $\Delta_d(n)>0$ for $n \geq 6$, divisible by $3$ and $d \geq 3$.
For each $n\geq 6$ not divisible by $3$,
we have an explicit bound $C(n)$,
such that $d \geq C(n)$
implies $\Delta_d(n)<0$.
\end{remarks}

\begin{theorem} \label{th12}
Let $n \geq 6$. Let $n \equiv 1$ or $n \equiv 2
\mod {3}$. 
Then $\Delta_d(n) < 0$ for almost all $d$. In particular
let
\begin{eqnarray*}
\tilde{C}_1(n)   &:= & 1 +  6\, \big( 1 + \ln (n-1)          \big) \,  n,  \\
\tilde{C}_2(n)  &:=  & 1 + 3\, \big(
3 + \ln (n+1) \big) \, n.
\end{eqnarray*}
Then for $r=1$ or $r=2$,
$\Delta_d(n)<0$ for $n \equiv r \mod {3}$ and $d \geq \tilde{C}_r(n)$.
\end{theorem}

In this paper we
study
the $n$ and $d$ aspects of $p_d(n)$.
For the rest of this section
we consider $p_d(n)$ as a double indexed sequence. Table \ref{landscape}
gives a
decent description of the underlying landscape.

Note that the results of column $d=1$ are due to Nicolas \cite{Ni78}.
He proved that for $n \geq 26$, there are no further exceptions.
Column $d=2$ was conjectured by Heim, Neuhauser, and
Tr\"oger \cite{HNT21} and proved by Ono, Pujahari, and Rolen \cite{OPR22}.
They proved that for $n \geq 12$, there are no further exceptions.
We also increased the values of the parameters $d$ and $n$.

\begin{conjecture}
Let $n\geq 6$ and $n \equiv 0 \mod {3}$. Then we expect
\begin{equation*}
\Delta_d(n) >0 \Longrightarrow \Delta_{d+1}(n)>0.
\end{equation*}
\end{conjecture}

\begin{conjecture}
Moreover, let $
{D}_{n}:= \min \left\{ d > 3\, : \, \Delta_d(n) <0 \right\} $.
Let $n \geq 6$ and $n \equiv 1 \mod{3}$. Then Table~\ref{landscape} indicates
$
{D}_{n} >
{D}_{n+1}$.
\end{conjecture}
\section{\label{abschnitt2}Ordinary and
plane partition functions}
According to Stanley, plane partitions are
fascinating generalizations of 
partitions of integers (\cite{St99}, Section 7.20). Andrews \cite{An98} gave an excellent introduction
of plane partitions in the context of higher-dimensional partitions. The most recent survey is presented
by Krattenthaler
on plane partitions in the work by
Stanley and his school
\cite{Kr16}. 

A plane partition $\pi$ of $n$ is an array 
$\pi = \left( \pi_{ij} \right)_{i,j \geq 1}$ of non-negative integers $\pi_{ij}$
with finite sum $\vert \pi \vert := \sum_{i,j\geq 1} \pi_{ij}=n$, which is weakly decreasing in rows and columns.

Plane partitions are also displayed by a filling of a Ferrers diagram with weakly decreasing rows and columns,
where the sum of all these numbers is equal to $n$.
Let the numbers in the filling represent the 
heights for stacks of blocks placed on each cell of the diagram (Figure \ref{cube}). 
\begin{figure}[H]
\begin{minipage}{0.35\textwidth}
\begin{equation*}
\phantom{xxx}
\young(5443321,432,21)
\end{equation*}
\end{minipage}
$\longrightarrow$ \phantom{xx}
\begin{minipage}{0.4\textwidth}\phantom{xx}
\includegraphics[width=0.75\textwidth]{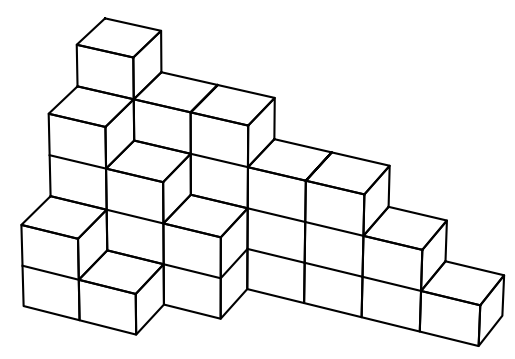}
\end{minipage}
\caption{\label{cube}Representation of plane partitions.}
\end{figure}

Let ${p}(n)$ denote the number of partitions of
$n$ and $\func{pp}(n)$ denote
the number of plane partitions of $n$.
As usual, ${p}\left( 0\right):=1$ and $\func{pp} \left(0 \right): =1$.
In Table \ref{partition} we have listed the first values
of the partition and plane partition function, which are equal to $p_1(n)$ and $p_2(n)$ and related to $d=1$ and $d=2$.
There had been speculations about the sequence $p_3(n)$. To indicate that this is still an open topic
we labeled $p_3(n)$ with $X_n$.
\begin{table}[H]
\[
\begin{array}{lrrrrrrrrrrr}
\hline
n & 0 & 1 & 2 & 3 & 4 & 5 & 6 & 7 & 8 & 9 & 10 \\ \hline \hline
p_1(n)=p\left( n \right) & 1 & 1 & 2 &
3 &
5 &
7 &
11 &
15 &
22 & 30 & 42 \\
p_2(n)=\func{pp} \left( n \right) & 1 & 1 & 3 &
6 &
13 &
24&
48 &
86 &
160 & 282 & 500 \\
p_3(n)= \,\, X_n \,\, & 1 & 1 &5&14&40&101&266&649&1593&3765&8813\\
\hline
\end{array}
\]
\caption{\label{partition}
Values for $ 0 \leq n \leq 10$.}
\end{table}
In the context of higher partitions \cite{An98} one has the so-called solid partitions
as certain  $3$-dimensional arrays. These are different from $p_3(n)$. 
MacMahon proposed possible generating
series, which
did not work out finally.

\section{Proofs of Theorem \ref{th03} and Theorem \ref{th12}}

We
divide
the proof into several parts.
Let $\sigma_{d}(n):= \sum_{
t \mid n}
t^d$.
Our strategy is to utilize the following formula (\cite{Ko04}, section 4.7):
\begin{equation*}
p_{d}(n) =\sum _{k\leq n}\sum _{\substack{m_{1},\ldots ,m_{k}\geq 1 \\ m_{1}+\ldots +m_{k}=n}}\frac{1}{k!}\frac{\sigma _{d}\left( m_{1}\right) \cdots \sigma _{d}\left( m_{k}\right) }{m_{1}\cdots m_{k}}.
\end{equation*}

\subsection{Partitions $n=\sum_{
j=1}^k m_{
j}$ 
and the maximum of $\prod_{
j=1}^k m_{
j}$}

We start with the following well known property.
For completeness, we
include
a proof.

\begin{lemma}\label{known}
\label{produkt}Let
$m_{1},\ldots ,m_{k}\geq 1$ and $m_{1}+\ldots +m_{k}=n\geq 2$. Then the largest
values for $m_{1}\cdots m_{k}$ are
\begin{eqnarray*}
\max _{k\leq n}\max _{\substack{m_{1},\ldots ,m_{k}\geq 1 \\ m_{1}+
\ldots +m_{k}=n}}m_{1}\cdots m_{k} &= &     3^{n/3}  \text{ for }  n      
\equiv 0\mod 3, \\
\max _{k\leq n}\max _{\substack{m_{1},\ldots ,m_{k}\geq 1 \\ m_{1}+
\ldots +m_{k}=n}}m_{1}\cdots m_{k} &= &          4\cdot 3^{\left( n-4\right) /3}  
\text{ for }  n   \equiv 1\mod 3,\\
\max _{k\leq n}\max _{\substack{m_{1},\ldots ,m_{k}\geq 1 \\ 
m_{1}+\ldots +m_{k}=n}}m_{1}\cdots m_{k} &=&      2\cdot 3^{\left( n-2\right) /3}  \text{ for } n\equiv 2\mod 3.
\end{eqnarray*}
\end{lemma}
\begin{proof}
Let $k\leq n$ and $m_{1},\ldots ,m_{k}\geq 1$ so
that
$m_{1}+\ldots +m_{k}=n$.
We can assume that $m_{j}\geq 2$ for all $j\leq k$.
Otherwise, if $m_{j}=1$, we can
take the partition with
$m_{j-1}$ replaced by $m_{j-1}+1$, or $m_{j+1}$ replaced by
$m_{j+1}+1$ and drop $m_{j}$ and obtain a larger product.
We can also assume, that $m_{j}\leq 4$ for all $j\leq k$.
Otherwise, if $m_{j}
>4$, we split it as
$m_{j}=2+\left( m_{j}-2\right) $ and
$2\left( m_{j}-2\right)
>m_{j}
$, resulting in a larger
product.

Therefore,
if $m_{1}\cdots m_{k}$ is maximal, then it follows
that
$2\leq m_{j}\leq 4$ for all $j\leq k$. Assume there are at
least three $m_{j}=2$. But $2+2+2=6=3+3$ and $2^{3}<3^{2}$.
So,
if we replace the three $2$s by two $3$s we obtain a
larger product. Therefore,
in the maximal product there can be at most two $2$s. Similarly,
if there were to be one
$m_{j}=4$ and another $=2$,
we obtain $4+2=6=3+3$ and again
$4\cdot 2<3^{2}$. Again,
the product becomes larger if
$4$ and $2$ are replaced by two $3$s. Therefore,
in the maximal
product there can only be either one $4$ and no $2$,
or no $4$
if there is a $2$.

In total,
we obtain that in the maximal product we have only
$3$s except for either precisely one $4$ or at most two $2$s.
Note, that if there is either
one $4$ or two $2$s, it
results in the same product as $4=2^{2}$. If $n\equiv 0\mod 3$,
this yields
only one possibility, when $n=\frac{n}{
3}
\cdot 3$ resulting in
$3^{n/3}$. If $n\equiv 1\mod 3$,
we have two possible
partitions:
$n=4+\frac{n-4}{3}\cdot 3=2+2+\frac{n-4}{3}\cdot 3$
where both
result
in the
product
$4\cdot 3^{\left( n-4\right) /3}$. Note that we
assumed $n\geq 2$. If $n\equiv 2\mod 3$,
we
again have only one possible partition $n=2+\frac{n-2}{3}\cdot 3$
given the restrictions,
resulting in the product
$2\cdot 3^{\left( n-2\right) /3
}$.
\end{proof}

\begin{lemma}
\label{zweitgroesstes}For $n\geq 8$, $n\equiv 2\mod 3$
the second largest product
is $16\cdot 3^{\left( n-8\right) /3}$.
\end{lemma}

\begin{proof}
In this
case
in the corresponding
partition $n=m_{1}+\ldots +m_{k}$ is
an $m_{j}=1$, an $m_{j}\geq 5$, more than one $m_{j}=4$,
at least three $m_{j}=2$, or a $4$ and a $2$ together.

First suppose there is an
$m_{j}=1$. Our argument is similar to that for
the proof of Lemma~\ref{known}
but in addition,
we have to ensure that the result is not maximal. We
replace either $m_{j-1}$ by $m_{j-1}+1$ or $m_{j+1}$ by
$m_{j+1}+1$ and drop the $m_{j}=1$. This can only result in
$n=2+\frac{n-2}{3}\cdot 3$ if either there were two $2$s in
the partition of $n$, or another $1$. If it was
$n=1+2+2+\frac{n-5}{3}\cdot 3$, then since
$n\geq 8$
we can change it to
$n=2+2+4+\frac{n-8}{3}\cdot 3$ and obtain a larger product.
In the second case, $n=1+1+\frac{n-2}{3}\cdot 3$.
Since $n\geq 8
$ we can
replace it by $n=1+4+\frac{n-5}{3}\cdot 3$ resulting in a
larger product.

In the second case we assume there is an $m_{j}\geq 5$. If we
again replace it by $2+\left( m_{j}-2\right) $ this will yield
the partition $n=2+\frac{n-2}{3}\cdot 3$, if and only if it was
$n=5+\frac{n-5}{3}\cdot 3$ before.
As $n\geq 8$ we can
increase this
to $n=4+4+\frac{n-8}{3}\cdot 3$,
since $4^{2}>5\cdot 3$.

In the third case we assume there
are more than one $4$. We
replace $4+4$ by $2+3+3$ and obtain a larger product
$2\cdot 3^{2}>4^{2}$. The product of the parts of
such a partition would only be maximal if the
partition was $n=4+4+\frac{n-8}{3}\cdot 3$.

The fourth case is that there are at least three $2$s. The product can be increased by
replacing $2+2+2=3+3$. This would
only be maximal if the partition was
$n=4\cdot 2+\frac{n-8}{3}\cdot 3$.

In the fifth case we assume that there is at least one $2$ and at least one $4$ together
in the partition. This could be increased by replacing $2+4=3+3$ as
$2\cdot 4<3^{2}$. This will yield the maximal product only if the partition was
$n=2+2+4+\frac{n-8}{3}$.

In total,
the
second largest product
$m_{1}\cdots m_{k}$ among all partitions
$\left( m_{1},\ldots ,m_{k}\right) $ of $n$
since
$n\geq 8$
is
\begin{equation}
16\cdot 3^{\left( n-8\right) /3}
\label{eq:2mod3}
\end{equation}
from the
partitions
$$n=4+4+\frac{n-8}{3}\cdot 3=2+2+4+\frac{n-8}{3}\cdot 3=4\cdot 2+\frac{n-8}{3}\cdot 3.$$
\end{proof}

\subsection{\label{schranken}Upper and
lower
bounds for $p_d(n)$}

Let $n\geq 2$ and
$d \geq 1
$. Then we have the inequalities
\begin{eqnarray*}
\frac{3^{\left( d-1\right) n/3}}{\left( n/3\right) !}
&<&p_{d}(n)\leq
3^{\left( d-1\right) n/3}p_{1}(n)\text{ for }n\equiv 0\mod 3,\\
\frac{3
\left( 4\cdot 3^{\left( n-4\right) /3}\right) ^{d-1}
}{2\left( \left( n-4\right) /3\right) !}
&<&p_{d}(n) \leq
\left( 4\cdot 3^{\left( n-4\right) /3}\right) ^{d-1}
p_{1}(n)\text{ for }n\equiv 1\mod 3,\\
\frac{\left( 2\cdot 3^{\left( n-2\right) /3}\right) ^{d-1}
}{\left( \left( n-2\right) /3\right) !}
&<&p_{d}(n) \leq
\left( 2\cdot 3^{\left( n-2\right) /3}\right) ^{d-1}p_{1}(n)\text{ for }n\equiv 2\mod 3.
\end{eqnarray*}
Moreover, let
$n \equiv 2 \mod 3$ and $
n\geq 8$.
Then 
\begin{equation} \label{improve}
p_{d}(n)\leq \frac{\left( 2\cdot 3^{\left( n-2\right) /3}\right) ^{d-1}}{\left( \left( n-2\right) /3\right) !}+
\left( 16\cdot 3^{\left( n-8\right) /3
}\right) ^{d-1}p_{1}(n).
\end{equation}

To show these we consider
\begin{eqnarray}
p_{d}\left( n\right)
&=&\sum _{k\leq n}\frac{1}{k!}\sum _{\substack{m_{1},\ldots ,m_{k}\geq 1 \\ m_{1}+
\ldots +m_{k}=n}}
\frac{\sigma _{d}\left( m_{1}\right) \cdots \sigma _{d}\left( m_{k}\right) }{m_{1}\cdots m_{k}}\nonumber \\
&=&\sum _{k\leq n}\frac{1}{k!}\sum _{\substack{m_{1},\ldots ,m_{k}\geq 1 \\ m_{1}+
\ldots +m_{k}=n}}
\frac{1}{m_{1}\cdots m_{k}}\left( \sum _{t_{1}\mid m_{1}}t_{1}^{d}\right) \cdots \left( \sum _{t_{k}\mid m_{k}}t_{k}^{d}\right) .\label{pdnalssumme}
\end{eqnarray}
For the upper bounds
we use the following estimate
\begin{eqnarray*}
p_{d}\left( n\right) &\leq &\sum _{k\leq n}\frac{1}{k!}\sum _{\substack{m_{1},\ldots ,m_{k}\geq 1 \\ m_{1}+
\ldots +m_{k}=n}}
\frac{1}{m_{1}\cdots m_{k}}\left( \sum _{t_{1}\mid m_{1}}t_{1}\right) \cdots \left( \sum _{t_{k}\mid m_{k}}t_{k}\right) \left( m_{1}\cdots m_{k}\right) ^{d-1}\\
&\leq
&p_{1}\left( n\right)
\max  _{k\leq n}\max _{\substack{m_{1},\ldots ,m_{k}\geq 1 \\ m_{1}+\ldots +m_{k}=n}}\left( m_{1}\cdots m_{k}\right) ^{d-1}.
\end{eqnarray*}
Using
Lemma~\ref{produkt} we obtain the upper bounds.
Similarly, we obtain the lower bounds by taking
only those
growth terms from
the previous sum (\ref{pdnalssumme}) where the
product is maximal. This we can obtain in the
following way. Let
$M=\max _{k\leq n}\max _{\substack{m_{1},\ldots ,m_{k}\geq 1 \\ m_{1}+\ldots +m_{k}=n}}m_{1}\cdots m_{k}$.
Continuing (\ref{pdnalssumme}) we then obtain
\[
p_{d}\left( n\right) \geq \sum _{k\leq n}\frac{1}{k!}\sum _{\substack{m_{1},\ldots ,m_{k}\geq 1 \\ m_{1}+\ldots +m_{k}=n \\ m_{1}\cdots m_{k}=M}}M^{d-1}.
\]
Note that for $n\equiv 1\mod 3$, we obtain the maximal value for $m_{1}=4$,
$m_{2}=\ldots =m_{\left( n-1\right) /3}=3$ and $m_{1}=m_{2}=2$,
$m_{3}=\ldots =m_{\left( n+2\right) /3}=3$, which contribute
$$\frac{3}{2\left( \left( n-4\right) /3\right) !}\left( 4\cdot 3^{\left( n-4\right) /3}\right) ^{d-1}$$
to the lower bound.

Now assume $n\equiv 2\mod 3$. Then we obtain
the improved upper bound (\ref{improve}). For this recall that
in the proof of Lemma~\ref{known}, we have shown that
$m_{1}\cdots m_{k}$ is maximal only if in the partition of
$n=m_{1}+\ldots +m_{k}$ there are only $3$s except for at most
either precisely one $4$ or at most two $2$s. Since presently
$n\equiv 2\mod 3$, the only possibility
is
$n=2+\frac{n-2}{3}\cdot 3$.

The rest follows using Lemma~\ref{zweitgroesstes}.

\subsection{\label{summary}Final
step in the
proofs of Theorem \ref{th03} and Theorem \ref{th12}.}

In this section
we prove
Theorem~\ref{th03} (first inequality
in Theorem~\ref{verbesserung}) and
Theorem~\ref{th12} (second and third inequality
in Theorem~\ref{verbesserung}).

\begin{thm12}
Let $n \geq 6$
be divisible by $3$. Then
$\Delta_d(n)>0$ for almost all $d$.
\end{thm12}

\begin{thm13}
Let $n \geq 6$. Let $n \equiv 1$ or $n \equiv 2
\mod {3}$.
Then $\Delta_d(n) < 0$ for almost all $d$. In particular
let
\begin{eqnarray*}
\tilde{C}_1(n)   &:= & 1 +  6\, \big( 1 + \ln (n-1)          \big) \,  n , \\
\tilde{C}_2(n)  &:=  & 1 + 3\, \big(
3 + \ln (n+1) \big) \, n.
\end{eqnarray*}
Then for $r=1$ or $r=2$, 
$\Delta_d(n)<0$ for $n \equiv r \mod {3}$ and $d \geq \tilde{C}_r(n)$.
\end{thm13}

Note that in Theorem~\ref{th12}, we can take the
least upper integer bound
$\left\lceil C\right\rceil $ of the constants
since $d$ is integer.
We show the following improvement of
Theorems~\ref{th03} and~\ref{th12}.

\begin{theorem}
\label{verbesserung}Let $n\geq 6$. Then:
\begin{eqnarray*}
\left( p_{d
}\left( n\right) \right) ^{2}&>&p_{d
}\left( n-1\right) p_{d
}\left( n+1\right) \text{ for }n\equiv 0\mod 3,\,d\geq C_{0}\left( n\right)
,\\
\left( p_{d
}\left( n\right) \right) ^{2}&<&p_{d
}\left( n-1\right) p_{d
}\left( n+1\right) \text{ for }n\equiv 1\mod 3,\, 
d \geq C_{1}\left( n\right)
,\\
\left( p_{d
}\left( n\right) \right) ^{2}&<&p_{d
}\left( n-1\right) p_{d
}\left( n+1\right) \text{ for }n\equiv 2\mod 3,
\, d\geq C_{2}\left( n\right)
\end{eqnarray*}
with
\begin{eqnarray*}
C_{0}\left( n\right) & := & 1+2\frac{
\ln \left( 2
\right) +\ln \left(
n                                                 /3\right) /3
}{\ln \left( 9/8\right) }n,\\
C_{1}\left( n\right) & := &    1+2\frac{
\ln \left( 2
\right) +\ln \left( \left( n-1\right) /3\right) /3
}{\ln \left( 9/8\right) }n,                     \\
C_{2}\left( n\right) & := &  1+\frac{\ln \left(
3
\right) +
\ln \left(
\left( n+1\right) /3\right) /3
}{\ln \left( 9/8\right) }
n.  
\end{eqnarray*}
\end{theorem}

\begin{proof}
We
apply the upper and lower bounds from
Subsection~\ref{schranken}. Let $n\geq 3$.

\noindent
\textbf{Case 1:}
$n\equiv 0\mod 3$.
Then
\begin{eqnarray*}
&&\frac{\left( p_{d
}\left( n\right) \right) ^{2}}{p_{d
}\left( n-1\right) p_{d
}\left( n+1\right) }\\
&\geq &
\frac{3^{\left( d-1\right) 2n/3}}{\left( \left( n/3\right) !\right) ^{2}p_{1
}\left( n-1\right)
\left( 2\cdot 3^{\left( n-3\right) /3}\right) ^{d-1}p_{1
}\left( n+1\right)
\left( 4\cdot 3^{
\left( n-3\right) /3}\right) ^{d-1}
}\\
&=&\frac{1}{\left( \left( n/3\right) !\right) ^{2}p_{1
}\left( n-1\right) p_{1
}\left( n+1\right)
}\left( \frac{9}{8}\right) ^{d-1}\\
&\geq &\left( n/3\right) ^{-2n/3}2^{-2n}\left( \frac{9}{8}\right) ^{d-1}
>1
\end{eqnarray*}
for
$d\geq 1+2\frac{
\ln \left( 2
\right) +\ln \left(
n
/3\right) /3
}{\ln \left( 9/8\right) }n$.

\noindent
\textbf{Case
2:}
$n\equiv 1\mod 3$.
Then
\begin{eqnarray*}
\frac{\left( p_{d
}\left( n\right) \right) ^{2}}{p_{d
}\left( n-1\right) p_{d
}\left( n+1\right) }&\leq &
\frac{\left( \left( n-1\right) /3\right) !\left( \left( n-1\right) /3\right) !\left(
\left( 4\cdot 3^{\left( n-4\right) /3}\right) ^{d-1}p_{1
}\left( n\right) \right) ^{2}
}{\left( 2\cdot 3^{\left( n-1\right) /3}\right) ^{d-1}3^{\left( n-1\right) \left( d-1\right) /3}}\\
&=&\left(
\left( \left( n-1\right) /3\right) !p_{1
}\left( n\right) \right) ^{2}\left( \frac{8}{9}\right) ^{d-1}\\
&\leq &
\left( \left( n-1\right) /3\right) ^{2n/3}2^{2n}\left( \frac{8}{9}\right) ^{d-1}
<1
\end{eqnarray*}
for
$d\geq 1+2\frac{
\ln \left( 2\right) +\ln \left( \left( n-1\right) /3\right) /3
}{\ln \left( 9/8\right) }n$.

\noindent \textbf{Case
3:}
$n\equiv 2\mod 3$ and $n\geq 11$.
Then
\begin{eqnarray*}
&&\frac{\left( p_{d
}\left( n\right) \right) ^{2}}{p_{d
}\left( n-1\right) p_{d
}\left( n+1\right) }\\
&\leq &\frac{\left( \frac{\left( 2\cdot 3^{\left( n-2\right) /3}\right) ^{d-1}}{\left( \left( n-2\right) /3\right) !}+
\left( 16\cdot 3^{\left( n-8\right) /3}\right) ^{d-1}p_{1
}\left( n\right) \right) ^{2}}{\frac{3\left( 4\cdot 3^{\left( n-5\right) /3}\right) ^{d-1}}{2\left( \left( n-5
\right) /3\right) !}\frac{3^{\left( d-1\right) \left( n+1\right) /3}}{\left( \left( n+1\right) /3\right) !}}\\
&=
&\frac{2\left( n+1\right) /3}{3\left( n-2\right) /
3}+\frac{
\frac{2}{\left( \left( n-2\right) /3\right) !}+
\left( \frac{8}{9}\right) ^{d-1}
p_{1
}\left( n\right)
}{\frac{3}{2\left( \left( n-5
\right) /3\right) !\left( \left( n+1\right) /3\right) !}}\left( \frac{8
}{
9
}\right) ^{d-1}p_{1
}\left( n\right) \\
&<
&\frac{8}{9}+
\frac{2
}{3
}
\left( 2+\left( \left( n+1\right) /3\right) ^{n/3}\left( \frac{8}{9}\right) ^{d-1}2^{n}\right) \left( \left( n+1\right) /3\right) ^{
n
/3}
\left( \frac{8}{9}\right) ^{d-1}2^{n}.
\end{eqnarray*}
If we suppose $d\geq C_{2}\left( n\right) $ we obtain
$\left( \left( n+1\right) /3\right) ^{n/3}\left(
{8}/{9}\right) ^{d-1}2^{n}\leq
\left( 2/3\right) ^{
n}
$
and
therefore,
\[
\frac{\left( p_{d}\left( n\right) \right) ^{2}}{p_{d}\left( n-1\right) p_{d}\left( n+1\right) }
<
\frac{8}{9}+\frac{2}{3}\left( 2+
\left( \frac{2}{3}\right) ^{
n}\right) \cdot
\left( \frac{2}{3}\right) ^{
n}<1.
\]

\noindent \textbf{Case
4:}
$n=8$. Here
the previous argument does not apply, since $\frac{2\left( n+1\right) }{3\left( n-2\right) }=1$.
Therefore, this case has to be treated separately.
To estimate $p_{d}\left( 8\right) $, we have to determine the third largest
products $m_{1}\cdots m_{k}$ when $m_{1}+\ldots +m_{k}=8$. The largest product is
$18=2\cdot 3\cdot 3$. From (\ref{eq:2mod3})
we obtain $16$ from
$4\cdot 4=4\cdot 2\cdot 2=
2^{4}$
as the second largest value. So, the third largest value could be
$15$, which is indeed obtained for
$5\cdot 3$. Therefore, we obtain
\[
p_{d}\left( 8\right) <\frac{1}{2}18^{d-1}+\frac{25}{24}16^{d-1}+21
\cdot 15^{d-1}
.
\]
In Lemma~\ref{pdn} we show that
\[
p_{d}\left( 7\right) >\frac{3}{2}12^{d-1}
.
\]
In a similar way we can
obtain that
\[
p_{d}\left( 9\right) >\frac{1}{6}27^{d-1}+\frac{7}{6}24^{d-1}
\]
from
$3\cdot 3\cdot 3$,
the largest product, and
$2\cdot 2\cdot 2\cdot 3=2\cdot 4\cdot 3$
the second largest products.
Therefore,
\begin{eqnarray*}
&&p_{d}\left( 7\right) p_{d}\left( 9\right) -\left( p_{d}\left( 8\right) \right) ^{2}\\
&>&\frac{17}{24}288^{d-1}
-21
\cdot 270^{d-1}-\frac{625}{576}256^{d-1}-\frac{
175}{4
}240^{d-1}
-441
\cdot 225^{d-1}\\
&>&\frac{17}{24}288^{d-1}
-508
\cdot 270^{d-1}>0
\end{eqnarray*}
for $d\geq 103
$. For $9\leq d\leq 102
$ it can be checked that also
$\frac{\left( p_{d}\left( 8\right) \right) ^{2}}{p_{d}\left( 7\right) p_{d}\left( 9\right) }<1$.
\end{proof}

\begin{remarks}
With
minor
simplifications (see below)
and using numerical values, we obtain also
the following lower bounds for $d$:
\begin{eqnarray*}
\left( p_{d
}\left( n\right) \right) ^{2}&>&p_{d
}\left( n-1\right) p_{d
}\left( n+1\right) \text{ for }n\equiv 0\mod 3,\, d\geq
C_{0}^{\ast }\left( n\right)
,\\
\left( p_{d
}\left( n\right) \right) ^{2}&<&p_{d              }\left( n-1\right) p_{d
}\left( n+1\right) \text{ for }n\equiv 1\mod 3,\, d\geq
C_{1}^{\ast }\left( n\right) ,\\
\left( p_{d
}\left( n\right) \right) ^{2}&<&p_{d
}\left( n-1\right) p_{d
}\left( n+1\right) \text{ for }n\equiv 2\mod 3,
\, d\geq
C_{2}^{\ast }\left( n\right)
\end{eqnarray*}
where we have
\begin{eqnarray*}
C_{0}^{\ast }\left( n\right) &:= &
1+ 5.67\left( 1+\ln \left( n\right) \right)
n,\\
C_{1}^{\ast }\left( n\right) &:= &
1+ 5.67\left( 1+\ln \left( n-1\right) \right)
n, \\
C_{2}^{\ast }\left( n\right)
&:= &
1+2.84\left( 2.2
+\ln \left( n+1\right) \right)
n.
\end{eqnarray*}
\end{remarks}

\begin{proof}[Simplifications]
Note that
$
\ln \left( n/3\right) /3=-\ln \left( 3\right) /3+\ln \left( n\right) /3$
and
\[
2\frac{\ln \left( 2\right) -\ln \left( 3\right) /3}{\ln \left( 9/8\right) }\approx 5.55<5.67
\]
and
$2\frac{1/3}{\ln \left( 9/8\right) }\approx 5.66<5.67$.
Therefore,
$5.67\left( 1+\ln \left( n\right) \right)
>2\frac{\ln \left( 2\right) +\ln \left( n/3\right) /3}{\ln \left( 9/8\right) }
$
which implies
$C_{r
}^{\ast }\left( n\right)
>C_{r
}\left( n\right) $
for $r
=0,1$. On the other hand
$\frac{\ln \left( 3
\right) -\ln \left( 3\right) /3}{\ln \left( 9/8\right) }\approx
6.22<
6.248$
and
$\frac{1/3}{\ln \left( 9/8\right) }\approx 2.83<2.84$.
Therefore,
$2.84\left( 2.2
+\ln \left( n+1\right) \right) >\frac{\ln \left( 3
\right) +\ln \left( \left( n+1\right) /3\right) /3}{\ln \left( 9/8\right) }
$
which implies
$C_{2
}^{\ast }\left( n\right) >C_{
2}\left( n\right) $.
\end{proof}

\section{Proof of Theorem \ref{EINS}}
We first express $p_d(n)$ by $p_d(n-1)$.
This
may be interesting
in its own
way.
\begin{lemma}
\label{pdn}Let $d \geq 1$. Then $p_{d}\left( 0
\right) = p_{d}\left( 1
\right) =1$ and $p_{d} (2)  = 2^{d-1}+      p_{d}(1)$.
Moreover, $p_{d}(3) =  3^{d-1}+p_{d}(2)$, $p_{d}(4)  = \frac{3}{2}\cdot 4^{d-1}+\frac{1}{2}2^{d-1}+p_{d}(3)$ and
\begin{eqnarray*}
p_{d}(5) &=&5^{d-1}+ 6^{d-1}+p_{d}(4),\\
p_{d}(6)&=&\frac{1}{2}9^{d-1}+\frac{7}{6}8^{d-1}+6^{d-1}+\frac{1}{2}4
^{
d-1
}+\frac{1}{2}3^{d-1}
+\frac{1
}{3}2^{d-1}+p_{d}(5),\\
p_{d}(7)&=&\frac{3
}{2
}12^{d-1}+10^{d-1}+7^{d-1}+\frac{1}{2}6^{d-1}+p_{d}(6),\\
p_{d}\left( 8\right) &=&\frac{1}{2
}18^{d-1}+\frac{25}{
24}16^{d-1}+15^{d-1}+12^{d-1}+\frac{7
}{4
}8^{d-1}+\frac{1}{2}6^{d-1} \\
& & {}+\frac{23
}{24}4^{d-1}+\frac{1}{4}2^{d-1}+p_{d}\left( 7\right) .
\end{eqnarray*}
\end{lemma}

\begin{corollary}
\label{logkonkav}
For $1\leq n
\leq 7
$ we have the following
relations between $\left( p_{d
}\left( n\right) \right) ^{2}$ and
$p_{d
}\left( n-1\right) p_{d
}\left( n+1\right) $:
\begin{eqnarray*}
\left( p_{d}\left( 1\right) \right) ^{2}&<&p_{d}\left( 0\right) p_{d}\left( 2\right) \text{ for all }d\geq 1, \\
\left( p_{d}\left( 2\right) \right) ^{2}&>&p_{d}\left( 1\right) p_{d}\left( 3\right) \text{ for all }d\geq 1, \\
\left( p_{d}\left( 3\right) \right) ^{2}&>&p_{d}\left( 2\right) p_{d}\left( 4\right) \text{
if and only if }d\geq 4, \\
\left( p_{d}\left( 4\right) \right) ^{2}&<&p_{d}\left( 3\right) p_{d}\left( 5\right) \text{
if and only if }d\geq 6, \\
\left( p_{d}\left( 5\right) \right) ^{2}&>&p_{d}\left( 4\right) p_{d}\left( 6\right) \text{
if and only if }d\geq 10, \\
\left( p_{d}\left( 6\right) \right) ^{2}&>&p_{d}\left( 5\right) p_{d}\left( 7\right) \text{ for all }d\geq 1, \\
\left( p_{d}\left( 7\right) \right) ^{2}&<&p_{d}\left( 6\right) p_{d}\left( 8\right) \text{
if and only if }d\leq 3\text{
or }d\geq 11.
\end{eqnarray*}
\end{corollary}

To complete the proof of Theorem~\ref{EINS}, Lemma~\ref{pdn} could be extended
to $n=9$ and $n=10$. Since the bounds
$C_{2}\left( 8\right) $ and
$1+18
\frac{\ln \left( 2\right) +\ln \left(
3\right) /3}{\ln \left( 9/8\right) }
$
are not too large, we can use the results of
Theorem~\ref{verbesserung}.

Assume $n=8$. Then $\left\lceil C_{2}\left( 8\right) \right\rceil =
101$.
It can be checked that
$\frac{\left( p_{d}\left( 8\right) \right) ^{2}}{p_{d}\left( 7\right) p_{d}\left( 9\right) }<
1$
for $9
\leq d
\leq
100$ and that
$\frac{\left( p_{d}\left( 8\right) \right) ^{2}}{p_{d}\left( 7\right) p_{d}\left( 9\right) }>1$
for $d\leq 8$.
For $n=9$ we obtain
$\left\lceil 1+18
\frac{\ln \left( 2\right) +\ln \left( 9/3\right) /3}{\ln \left( 9/8\right) }\right\rceil =163$.
Again, it can be checked that
$\frac{\left( p_{d}\left( 9\right) \right) ^{2}}{p_{d}\left( 8\right) p_{d}\left( 10\right) }>1$
for all $3\leq d
\leq 162
$ and that
$\frac{\left( p_{d}\left( 9\right) \right) ^{2}}{p_{d}\left( 8\right) p_{d}\left( 10\right) }<1$
for $d\leq 2$.

\appendix

\section{Bounds on $d$}

\begin{figure}[H]
\includegraphics[width=0.4\textwidth]{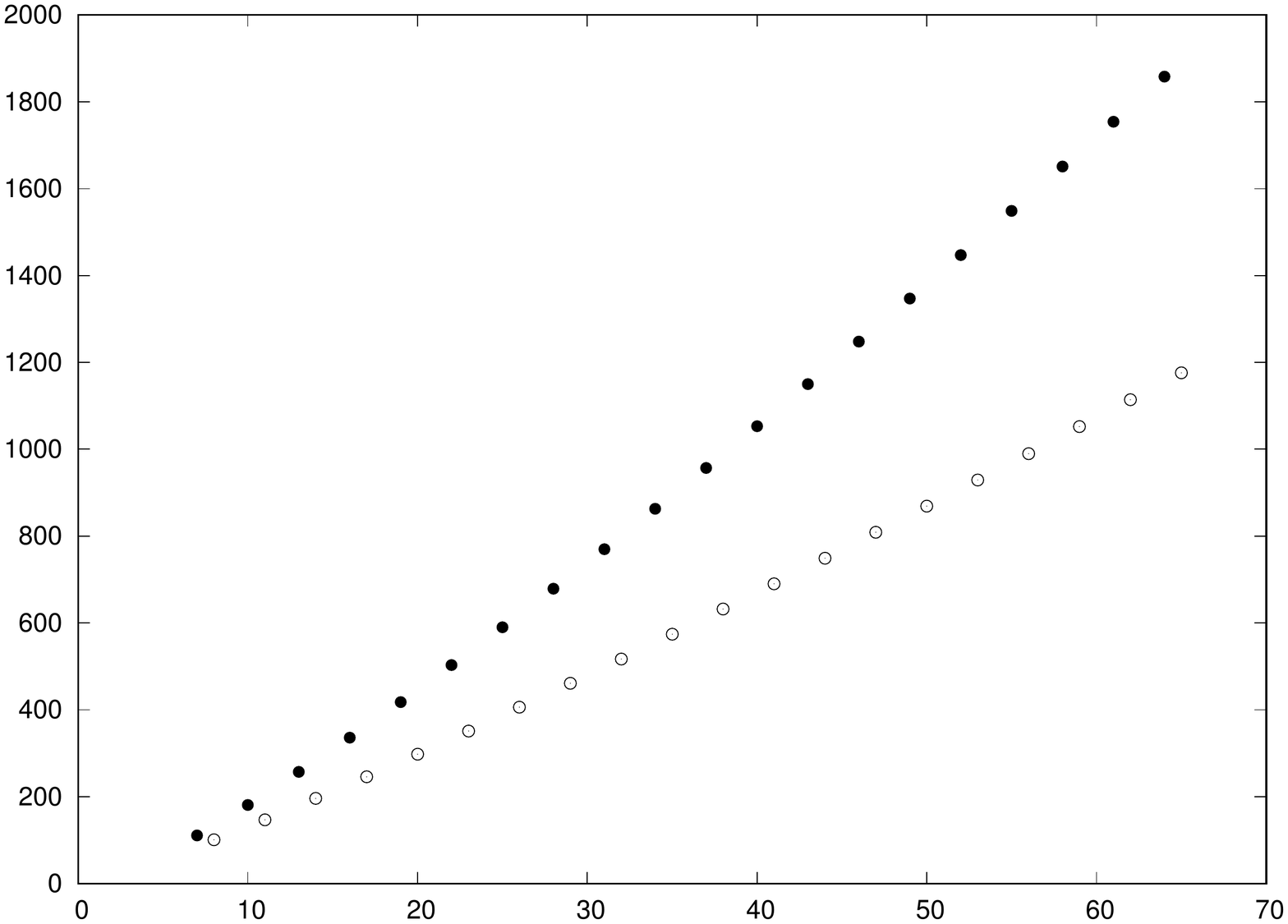}
\includegraphics[width=0.4\textwidth]{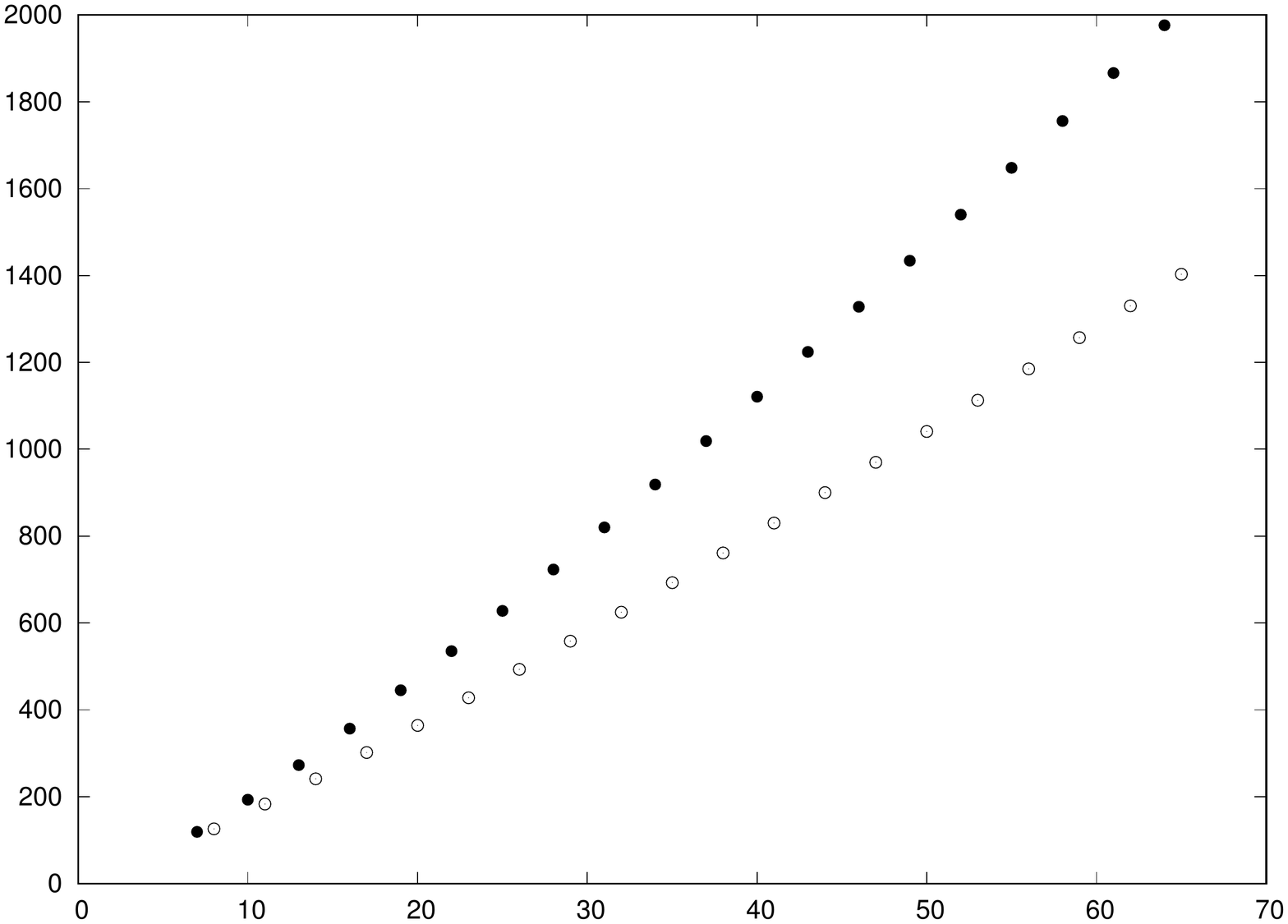}
\caption{\label{bild}Graphical representation of the bounds on $d$ from Table~\ref{vergleich}: left
$C_{1}\left( n\right) ,C_{2}\left( n\right) $, right
$\tilde{C}_{1}\left( n\right) ,\tilde{C}_{2}\left( n\right) $.}
\end{figure}

\begin{table}[H]
\[
\begin{array}{rrrrrr}
\hline
n&
\left\lceil C_{1}\left( n\right) \right\rceil
& \left\lceil \tilde{C}_{1}\left( n\right) \right\rceil &n&
\left\lceil C_{2}\left( n\right) \right\rceil
&\left\lceil \tilde{C}_{2}\left( n\right) \right\rceil \\ \hline \hline
7 & 111 & 119 & 8 & 101 & 126 \\
10 & 181 & 193 & 11 & 147 & 183 \\
13 & 257 & 273 & 14 & 196 & 241 \\
16 & 336 & 357 & 17 & 246 & 302 \\
19 & 418 & 445 & 20 & 298 & 364 \\
22 & 503 & 535 & 23 & 351 & 428 \\
25 & 590 & 628 & 26 & 406 & 493 \\
28 & 679 & 723 & 29 & 461 & 558 \\
31 & 770 & 820 & 32 & 517 & 625 \\
34 & 863 & 919 & 35 & 574 & 693 \\
37 & 957 & 1019 & 38 & 632 & 761 \\
40 & 1053 & 1121 & 41 & 690 & 830 \\
43 & 1150 & 1224 & 44 & 749 & 900 \\
46 & 1248 & 1328 & 47 & 809 & 970 \\
49 & 1347 & 1434 & 50 & 869 & 1041 \\
52 & 1447 & 1540 & 53 & 929 & 1113 \\
55 & 1549 & 1648 & 56 & 990 & 1185 \\
58 & 1651 & 1756 & 59 & 1052 & 1257 \\
\hline
\end{array}
\]
\caption{\label{vergleich}Comparison of the bounds on $d$.}
\end{table}

\section{Proofs of Lemma~\ref{pdn} and Corollary~\ref{logkonkav}}

\begin{proof}[Proof of Lemma~\ref{pdn}]
We have $p_d(0)= p_d(1)=1$. Moreover $p_d(2) = \frac{\sigma _{d}\left( 2\right) +1}{2}=2^{d-1}+1$. Further,
$p_{d}(3)=\frac{\sigma _{d}\left( 3\right) }{3}+\frac{
\sigma _{d}\left( 2\right) }{2}+\frac{1}{6}=3^{d-1}+2^{d-1}+1=3^{d-1}+p_{d}(2)$. Moreover,
\begin{eqnarray*}
p_{d}(4)&=&\frac{\sigma _{d}\left( 4\right) }{4}+\frac{\sigma _{d}\left( 3\right) }{3}+\frac{\left( \sigma _{d}\left( 2\right) \right) ^{2}}{8}+\frac{\sigma _{d}\left( 2\right) }{4}+\frac{1}{24}
=\frac{3}{2}\cdot 4^{d-1}+\frac{1}{2}2^{d-1}+p_{d}(3),\\
p_{d}(5)&=&\frac{\sigma _{d}\left( 5\right) }{5}+\frac{\sigma _{d}\left( 4\right) }{4}+\frac{\sigma _{d}\left( 3\right) \sigma _{d}\left( 2\right) }{6}+\frac{\sigma _{d}\left( 3\right) }{6}+\frac{\left( \sigma _{d}\left( 2\right) \right) ^{2}}{8}+\frac{\sigma _{d}\left( 2\right) }{12}+\frac{1}{120}\\                                                 &=&6^{d-1}+5^{d-1}+\frac{3}{2}\cdot               4                                                 ^{                                                d-1                                               }+                                                3^{d-1}+\frac{3}{2}\cdot 2^{d-1                   }                                                 +1=6^{d-1}+5^{d-1}+                               p_{d}(4),\\
p_{d}(6)&=&\frac{\sigma _{d}\left( 6\right) }{6}+\frac{\sigma _{d}\left( 5\right) }{5}+\frac{\sigma _{d}\left( 4\right) \sigma _{d}\left( 2\right) }{8}+\frac{\sigma _{d}\left( 4\right) }{8}+\frac{\left( \sigma _{d}\left( 3\right) \right) ^{2}}{18}+\frac{\sigma _{d}\left( 3\right) \sigma _{d}\left( 2\right) }{6}\\                                    &&{}+\frac{\sigma _{d}\left( 3\right) }{18}+\frac{\left( \sigma _{d}\left( 2\right) \right) ^{3}}{48}+\frac{\left( \sigma _{d}\left( 2\right) \right) ^{2}}{16}+\frac{\sigma _{d}\left( 2\right) }{48}+\frac{1}{720}\\                                    &=&\frac{1}{2                                     }9^{d-1}+\frac{7}{6                               }8^{d-1}+2\cdot 6^{d-1}+5^{d-1}+2\cdot 4^{d-1}+\frac{3}{2}3^{d-1}+\frac{11}{6}2^{d-1}+1\\           &=&\frac{1}{2                                     }9^{d-1}+\frac{7}{6
}8^{d-1}+6^{d-1}+\frac{1}{2}4                     ^{                                                d-1                                               }+\frac{1}{2}3^{d-1}                              +\frac{1                                          }{3                                               }2^{d-1}+p_{d}(5),\\                              p_{d}(7)&=&\frac{\sigma _{d}\left( 7\right) }{7}+\frac{\sigma _{d}\left( 6\right) }{6}+\frac{\sigma _{d}\left( 5\right) \sigma _{d}\left( 2\right) }{10}+\frac{\sigma _{d}\left( 5\right) }{10}+\frac{\sigma _{d}\left( 4\right) \sigma _{d}\left( 3\right) }{12}\\
&&{}+\frac{\sigma _{d}\left( 4\right) \sigma _{d}\left( 2\right) }{8}+\frac{\sigma _{d}\left( 4\right) }{24}+\frac{\left( \sigma _{d}\left( 3\right) \right) ^{2}}{18}+\frac{\sigma _{d}\left( 3\right) \left( \sigma _{d}\left( 2\right) \right) ^{2}}{24}\\                                               &&{}+\frac{\sigma _{d}\left( 3\right) \sigma _{d}\left( 2\right) }{12}+\frac{\sigma _{d}\left( 3\right) }{72}+\frac{\left( \sigma _{d}\left( 2\right) \right) ^{3}}{48}+\frac{\left( \sigma _{d}\left( 2\right) \right) ^{2}}{48}+\frac{\sigma _{d}\left( 2\right) }{240}+\frac{1}{5040}\\                  &=&\frac{3                                        }{2                                               }12^{d-1}+10^{d-1}+\frac{1}{2}9^{d-1}+\frac{7}{6}8^{d-1}+7^{d-1}+\frac{5}{2}6^{d-1}+5^{d-1}+2\cdot 4^{d-1}\\                                          &&{}+\frac{3
}{2}3^{d-1}+\frac{11}{6                           }2^{d-1}+1 =                                      \frac{3}{2}12^{d-1}+10^{d-1}+7^{d-1}+\frac{1}{2}6^{d-1}+p_{d}(6),                                   \end{eqnarray*}                                   and finally                                       \begin{eqnarray*}                                 p_{d}\left( 8\right) &=&\frac{\sigma _{d}\left( 8\right) }{8}+\frac{\sigma _{d}\left( 7\right) }{7}+\frac{\sigma _{d}\left( 6\right) \sigma _{d}\left( 2\right) }{12}+\frac{\sigma _{d}\left( 6\right) }{12}+\frac{\sigma _{d}\left( 5\right) \sigma _{d}\left( 3\right) }{15}\\                            &&{}+\frac{\sigma _{d}\left( 5\right) \sigma _{d}\left( 2\right) }{10}+\frac{\sigma _{d}\left( 5\right) }{30}+\frac{\left( \sigma _{d}\left( 4\right) \right) ^{2}}{32}+\frac{\sigma _{d}\left( 4\right) \sigma _{d}\left( 3\right) }{12}\\
&&{}+\frac{\sigma _{d}\left( 4\right) \left( \sigma _{d}\left( 2\right) \right) ^{2}}{32}+\frac{\sigma _{d}\left( 4\right) \sigma _{d}\left( 2\right) }{16}+\frac{\sigma _{d}\left( 4\right) }{96}+\frac{\left( \sigma _{d}\left( 3\right) \right) ^{2}\sigma _{d}\left( 2\right) }{36}\\
&&{}+\frac{\left( \sigma _{d}\left( 3\right) \right) ^{2}}{36}+\frac{\sigma _{d}\left( 3\right) \left( \sigma _{d}\left( 2\right) \right) ^{2}}{24}+\frac{\sigma _{d}\left( 3\right) \sigma _{d}\left( 2\right) }{36}+\frac{\sigma _{d}\left( 3\right) }{360}+\frac{\left( \sigma _{d}\left( 2\right) \right) ^{4}}{384}\\
&&{}+\frac{\left( \sigma _{d}\left( 2\right) \right) ^{3}}{96}+\frac{\left( \sigma _{d}\left( 2\right) \right) ^{2}}{192}+\frac{\sigma _{d}\left( 2\right) }{1440}+\frac{1}{40320}\\
&=&\frac{1}{2}18^{d-1}+\frac{25}{
24}16^{d-1}+15^{d-1}+\frac{5}{2}12^{d-1}+10^{d-1} +\frac{1}{2}9^{d-1}+\frac{35}{12
}8^{d-1}\\
&&{}
+7^{d-1}
+3\cdot
6^{d-1}+5^{d-1}+\frac{71}{24}4^{d-1}+\frac{3}{2}3^{d-1}+\frac{25}{12
}2^{d-1}+1\\
&
=&
\frac{1}{2
}18^{d-1}+\frac{25}{                           
24}16^{d-1}
+15^{d-1}+12^{d-1}+\frac{7}{4
}8^{d-1}+\frac{1}{2}6^{d-1}+\frac{23
}{24
}4^{d-1}+\frac{1}{4
}2^{d-1}\\
&&{}
+p_{d}\left( 7\right) .\\
\end{eqnarray*}
\end{proof} 

\begin{proof}[Proof of Corollary~\ref{logkonkav}]
The property $\Delta_d(1)<
0$ follows from $p_d(0)=p_d(1)=1$ and $p_d(2)= 2^{d-1} +1 >1$ for $d \geq 1$.
We use the results of Lemma \ref{pdn} to estimate the $p_{d}\left( n\right) $.
Therefore,  $p_{d}(2)\geq 2^{d-1}$ and
$p_{d}(3) \leq 3^{d}$. Thus,
\[
\frac{p_{d}(2)^{2}}{p_{d}(1) \, p_{d}(3) }\geq \frac{4^{d-1}}{3^{d}}>1
\]
for $d\geq 5$.
For 
the remaining cases $1\leq d\leq 4$
we checked directly that
$\left( p_{d}\left( 2
\right) \right) ^{2}
>p_{d}\left( 1
\right) p_{d}\left( 3
\right) $.

We have $p_{d}\left( 3\right) \geq 3^{d-1}$, $p_{d}\left( 2\right) \leq 2^{d}$,
and $p_{d}\left(
4
\right) \leq 5\cdot 4^{d-1}$. Therefore,
\[
\frac{\left( p_{d
}\left( 3\right) \right) ^{2}}{p_{d
}\left( 2\right) p_{d
}\left( 4\right) }\geq \frac{9^{d-1}}{10\cdot 8^{d-1}}>1
\]
for
$d\geq 21$.
For $4
\leq d\leq 20$ it can be checked that 
$\left( p_{d}\left( 3
\right) \right) ^{2}>p_{d}\left( 2
\right) p_{d}\left( 4
\right) $ and for $d\leq 3$ that
$\left( p_{d}\left( 3\right) \right) ^{2}<p_{d}\left( 2\right) p_{d}\left( 4\right) $.

From
Lemma~\ref{pdn}, we obtain
$p_{d}\left( 3
\right) p_{d}\left( 5
\right) \geq 18^{d-1}$ and
$p_{d}\left( 4\right) ^{2}\leq \left( 5
\cdot 4^{d-1}\right) ^{2}=25\cdot 16^{d-1}$. Then
\[
\frac{p_{d}\left( 3\right) p_{d}\left(
5\right) }{
\left( p_{d}\left( 4\right) \right) ^2}\geq \frac{18^{d-1}}{
25\cdot 16^{d-1}}>1
\]
for $d\geq 29$. For
$6\leq d\leq 28$ we can check that
$\frac{p_{d}\left( 3
\right) p_{d}\left(
5\right)
}{\left( p_{d}\left( 4\right) \right) ^{2}}>1
$ and for $d\leq 5$ that
$\frac{p_{d}\left( 3
\right) p_{d}\left( 5
\right) }{\left( p_{d}\left( 4\right) \right) ^{2}}<1$.

We have $p_{d}\left(
5
\right) \geq
6^{d-1}
$, $p_{d}\left( 4\right) \leq \frac{3}{2}4^{d-1}+\frac{7}{2}3^{d-1}$, and
$p_{d}\left( 6\right) \leq \frac{1}{2}9^{d-1}+\frac{21}{2}8^{d-1}
$. Therefore,
\[
\frac{\left( p_{d}\left( 5\right) \right) ^{2}}{p_{d}\left( 4\right) p_{d}\left( 6\right) }\geq \frac{36^{d-1}}{\frac{3}{4}36^{d-1}+\frac{63
}{4}32^{d-1}+\frac{7}{4}27^{d-1}+\frac{147
}{4}24^{d-1}}>\frac{36^{d-1}}{\frac{3}{4}36^{d-1}+55
\cdot 32^{d-1}}\geq 1
\]
for $d\geq 47
$. For $10
\leq d\leq 46
$ it can be checked that
$\left( p_{d}\left( 5
\right) \right) ^{2}
>p_{d}\left(
4\right) p_{d}\left(
6\right) $ and for $
d\leq 9$ that
$\left( p_{d}\left( 5\right) \right) ^{2}<
p_{d}\left( 4\right) p_{d}\left( 6\right) $.

We have $p_{d}\left( 6\right) \geq \frac{1}{2}9^{d-1}$,
$p_{d}\left( 5\right) \leq
7\cdot 6^{d-1}$, and
$p_{d}\left( 7\right) \leq 15\cdot 12^{d-1}$. Therefore,
\[
\frac{\left( p_{d}\left( 6\right) \right) ^{2}}{p_{d}\left( 5\right) p_{d}\left( 7\right) }\geq \frac{81^{d-1}}{420\cdot 72^{d-1}}>1
\]
for $d\geq 53$. For $
d\leq 52$ it can be checked that
$\left( p_{d}\left( 6\right) \right) ^{2}>p_{d}\left( 5\right) p_{d}\left( 7\right) $.

We have $p_{d}\left( 7\right) \leq 15\cdot 12^{d-1}$,
$p_{d}\left( 6\right) >\frac{1}{2}9^{d-1}$, and
$p_{d}\left( 8\right) >\frac{1}{2}18^{d-1}$. Therefore,
\[
\frac{\left( p_{d}\left( 7\right) \right) ^{2}}{p_{d}\left( 6\right) p_{d}\left( 8\right) }<\frac{225\cdot 144^{d-1}}{\frac{1}{4}162^{d-1}}<1
\]
for $d\geq 59$. For $d\leq 3$ and $11\leq d
\leq 58$ it can also be checked that
$\frac{\left( p_{d}\left( 7\right) \right) ^{2}}{p_{d}\left( 6\right) p_{d}\left( 8\right) }<1$
and for $4\leq d\leq 10$ that
$\frac{\left( p_{d}\left( 7\right) \right) ^{2}}{p_{d}\left( 6\right) p_{d}\left( 8\right) }>1$.
\end{proof}

\section{Patterns in the $n$ and $d$ directions}
\begin{table}[H]                                  \[                                                \begin{array}{r|cccccccccccccccccccc}             \hline                                            n\backslash d&{1}& {2}&3&4&5&6&7&8&9&10&11&12&13&14&15&16&17&18&19&20\\ \hline \hline               1&\bullet &\bullet &\bullet &\bullet &\bullet &\bullet &\bullet &\bullet &\bullet &\bullet &\bullet &\bullet &\bullet &\bullet &\bullet &\bullet &\bullet &\bullet &\bullet &\bullet  \\                2&&&&&&&&&&&&&&&&&&&&\\                           3&\bullet &\bullet &\bullet &&&&&&&&&&&&&&&&&\\   \hline                                            4&&&&&&\bullet &\bullet &\bullet &\bullet &\bullet &\bullet &\bullet &\bullet &\bullet &\bullet &\bullet &\bullet &\bullet &\bullet &\bullet \\       5&\bullet &\bullet &\bullet &\bullet &\bullet &\bullet &\bullet &\bullet &\bullet &&&&&&&&&&&\\
6&&&&&&&&&&&&&&&&&&&&\\                           \hline                                            7&\bullet &\bullet &\bullet &&&&&&&&\bullet &\bullet &\bullet &\bullet &\bullet &\bullet &\bullet &\bullet &\bullet &\bullet  \\                      8&&&&&&&&&\bullet &\bullet &\bullet &\bullet &\bullet &\bullet &\bullet &\bullet &\bullet &\bullet &\bullet &\bullet  \\                              9&\bullet &\bullet &&&&&&&&&&&&&&&&&&\\           \hline                                            10&&&&&&&&&&&&&&&&\bullet &\bullet &\bullet &\bullet &\bullet \\                                    11&\bullet &\bullet &&&&&&&&&&\bullet &\bullet &\bullet &\bullet &\bullet &\bullet &\bullet &\bullet &\bullet  \\                                     12&&&&&&&&&&&&&&&&&&&&\\                          \hline                                            13&\bullet &&&&&&&&&&&&&&&&&&&\bullet \\
14&&&&&&&&&&&&&&&\bullet &\bullet &\bullet &\bullet &\bullet &\bullet  \\                           15&\bullet &&&&&&&&&&&&&&&&&&&\\                  \hline                                            16&&&&&&&&&&&&&&&&&&&&\\                          17&\bullet &&&&&&&&&&&&&&&&&\bullet &\bullet &\bullet \\                                            18&&&&&&&&&&&&&&&&&&&&\\                          \hline                                            19&\bullet &&&&&&&&&&&&&&&&&&&\\                  20&&&&&&&&&&&&&&&&&&&&\bullet  \\                 21&\bullet &&&&&&&&&&&&&&&&&&&\\                  \hline                                            22&&&&&&&&&&&&&&&&&&&&\\                          23&\bullet &&&&&&&&&&&&&&&&&&& \\                 24&&&&&&&&&&&&&&&&&&&&\\                          \hline                                            25&\bullet &&&&&&&&&&&&&&&&&&&\\
26&&&&&&&&&&&&&&&&&&&&\\
\end{array}
\]
\caption{\label{landscape}Exceptions for $1\leq d\leq 20$ and $1\leq n\leq 26$.}
\end{table}

\begin{Acknowledgments}
We thank George Andrews and Ken Ono for encouraging feedback on the significance of the topic.
We further thank the two referees for carefully checking the calculations performed in the paper 
and their helpful comments.
\end{Acknowledgments}

\end{document}